\documentclass[10pt,reqno]{amsart}

\usepackage[T1]{fontenc}
\usepackage{lmodern}
\usepackage{amsmath,amssymb,amsthm,mathtools}
\usepackage{enumitem}
\usepackage{microtype}
\usepackage{xcolor}
\usepackage[colorlinks=true,linkcolor=blue,citecolor=blue,urlcolor=blue]{hyperref}
\usepackage[capitalize,nameinlink]{cleveref}
\usepackage{float}

\newtheorem{theorem}{Theorem}[section]
\newtheorem{proposition}[theorem]{Proposition}
\newtheorem{lemma}[theorem]{Lemma}
\newtheorem{corollary}[theorem]{Corollary}
\theoremstyle{definition}
\newtheorem{definition}{Definition}[section]
\theoremstyle{remark}

\numberwithin{equation}{section}

\newcommand{\Cl}{\mathcal C}

\newcommand{\lcm}{\operatorname{lcm}}

\title[Reduction of Zieve's conjecture]
{Reduction of Zieve's Conjecture on the Hurwitz Existence Problem to Three Branch Points}

\newcommand{\ustcMathAddress}{School of Mathematical Sciences,
  University of Science and Technology of China, Hefei, Anhui 230026, China}
\newcommand{\hefeiLabAddress}{Hefei National Laboratory,
  University of Science and Technology of China, Hefei, Anhui 230088, China}
\newcommand{\ustcInfoAddress}{National Engineering Laboratory for
  Brain-inspired Intelligence Technology and Applications,
  School of Information Science and Technology,
  University of Science and Technology of China, Hefei, Anhui 230027, China}

\author[X. Liu]{Xitaiyu Liu\ensuremath{^{\mathrm{a},\dagger}}}
\address{\ustcMathAddress}
\email{lxty@mail.ustc.edu.cn}
\thanks{\(\dagger\) Corresponding author: Xitaiyu Liu
  (\texttt{lxty@mail.ustc.edu.cn}).}

\author[T. Sun]{Tianyang Sun\ensuremath{^{\mathrm{a}}}}
\address{\ustcMathAddress}
\email{tysun@mail.ustc.edu.cn}

\author[B. Xu]{Bin Xu\ensuremath{^{\mathrm{a}}}}
\address{\ustcMathAddress}
\email{bxu@ustc.edu.cn}

\author[Y. Ye]{Yu Ye\ensuremath{^{\mathrm{a},\mathrm{b}}}}
\address{\ustcMathAddress}
\address{\hefeiLabAddress}
\email{yeyu@ustc.edu.cn}

\author[Y. Zhou]{Yi Zhou\ensuremath{^{\mathrm{c}}}}
\address{\ustcInfoAddress}
\email{yi\_zhou@ustc.edu.cn}

\AddToHook{cmd/@setauthors/after}{%
  \begin{center}
    \footnotesize
    \({}^{\mathrm{a}}\)\,\ustcMathAddress\par
    \({}^{\mathrm{b}}\)\,\hefeiLabAddress\par
    \({}^{\mathrm{c}}\)\,\ustcInfoAddress
  \end{center}%
}


\begin{document}

\begin{abstract}
M. E. Zieve conjectured that a candidate datum over the sphere is
realizable whenever the gcd of the parts of every branch partition is one and
it satisfies the lcm condition.  We prove that this
conjecture reduces to its
three-branch-point case.  More precisely, if every Zieve-admissible triple is
realizable, then every Zieve-admissible datum with at least four branch points
is realizable. As already noted by Petronio, Zieve's conjecture would imply
the prime-degree conjecture of Edmonds, Kulkarni, and Stong. Song, Wen, and Zhang proved the latter conjecture in a preprint posted on arXiv in September.
\end{abstract}

\maketitle
%\tableofcontents

\section{Introduction}\label{sec:introduction}

The Hurwitz existence problem asks which combinatorial branch data arise from
branched coverings between closed surfaces. Throughout this paper, we consider
connected closed orientable surfaces and take the target to be \(S^2\).
For a degree-\(d\) cover of the
sphere with branch points \(q_1,\ldots,q_k\), the local degrees above \(q_i\)
form a partition \(\lambda_i\vdash d\).  The Riemann--Hurwitz formula gives a
necessary numerical condition.  By monodromy, realizability is equivalent to
the existence of a transitive product-one tuple of permutations having the
prescribed cycle types;
see Edmonds, Kulkarni, and Stong \cite{EKS1984} and Petronio
\cite{Petronio2020} for the general framework.

The Riemann--Hurwitz condition, though necessary, is far from sufficient in general. Edmonds, Kulkarni, and Stong developed a systematic permutation-theoretic approach and proved several fundamental realizability and reduction results \cite{EKS1984}. A spherical target datum is called {\it exceptional} if it satisfies the Riemann--Hurwitz condition but cannot be realized by any branched cover to the sphere.
Zheng \cite{Zheng2006} introduced an algorithm and used it on his personal computer to enumerate all such exceptional data consisting of three partitions and having degree less than 21. More recently, F. Baroni and C. Petronio completed the classification in the case where one branch partition has length two \cite{BaroniPetronio2024}. In the same work, they also identified all exceptional data of degree less than 30 using a computing server. Subsequently, Y. Wang et al. \cite{WX2025} extended this enumeration to all degrees less than 32 using a supercomputer at USTC.
Nevertheless, broad families in which all partitions have larger length remain to be classified.

Petronio's survey records a conjecture communicated by Zieve in 2019
\cite[Section~3.9]{Petronio2020}.  In the orientable spherical-target setting,
it predicts that every Zieve-admissible datum is realizable.
The conjecture is still open even
for three branch points.  The purpose of this paper is not to resolve that
base case, but to show that no additional obstruction occurs when more branch
points are allowed.  As explained in the same survey, Zieve's
conjecture would also imply the Edmonds--Kulkarni--Stong prime-degree
conjecture, which asserts that every candidate datum of prime degree is
realizable. In September 2026, J. Song, H. Wen, and Z. Zhang \cite{SWZ2026} proved the prime-degree conjecture using methods in positive characteristic.

\begin{definition}[Zieve-admissible datum]
A degree-\(d\) \emph{branch datum} is a finite multiset of partitions of
\(d\), written \(\Lambda=(\lambda_1,\ldots,\lambda_k)\) for an indexed
representative; the ordering is not part of the datum. For
\(\lambda=[a_1,\ldots,a_r]\vdash d\), put \(\ell(\lambda)=r\),
\(b(\lambda)=d-r\), \(g(\lambda)=\gcd(a_1,\ldots,a_r)\), and
\(m(\lambda)=\lcm(a_1,\ldots,a_r)\). Here \(b(\lambda)\) is the branching
defect, and \(\lambda\) is \emph{nontrivial} if \(\lambda\ne[1^d]\).
The datum \(\Lambda\) is a \emph{candidate datum} if every entry is
nontrivial and
\[
  \sum_{i=1}^k b(\lambda_i)=2(d+h-1)
\]
for some integer \(h\ge0\), called its \emph{source genus}. Its
\emph{lcm condition} is
\[
  \sum_{i=1}^k\left(1-\frac1{m(\lambda_i)}\right)\ne2.
\]
A candidate datum is \emph{Zieve-admissible} if it is branchwise gcd-one
(\(g(\lambda_i)=1\) for every \(i\)) and satisfies the lcm condition.
\end{definition}

\begin{definition}[Realizability]
A branch datum \(\Lambda=(\lambda_1,\ldots,\lambda_k)\) is
\emph{realizable} if there exist permutations \(\sigma_i\in\mathfrak S_d\)
of cycle types \(\lambda_i\) such that \(\sigma_1\cdots\sigma_k=1\) and
\(\langle\sigma_1,\ldots,\sigma_k\rangle\) acts transitively on
\(\{1,\ldots,d\}\).
\end{definition}

Hurwitz moves show that this condition is independent of the chosen indexing.

The main result can now be stated as follows.

\begin{theorem}[Main theorem]\label{thm:main}
Fix \(d\ge2\).  If every Zieve-admissible triple of degree \(d\) is realizable,
then every Zieve-admissible datum of degree \(d\) with
\(k\) branches is realizable for every \(k\ge4\).
\end{theorem}

For \(d=2\), the only nontrivial partition is \([2]\), whose gcd is \(2\).
Thus there is no Zieve-admissible datum of degree \(2\), and the theorem is
vacuous in that degree.  In the reduction below we may therefore assume
\(d\ge3\).

We next outline the reduction.  The formal statements
and complete constructions appear in Sections~\ref{sec:preliminaries}--
\ref{sec:residual-zero}.

\begin{definition}[Reduction]
For \(\lambda\vdash d\), let \(\Cl_\lambda\subset\mathfrak S_d\) be the
conjugacy class of permutations of cycle type \(\lambda\).
A \emph{reduction} replaces two entries \(\lambda,\nu\)
of a datum by a partition \(M\vdash d\) satisfying
\(\Cl_M\subset\Cl_\lambda\Cl_\nu\). We call \(M\) the \emph{output}
and the resulting datum the \emph{reduced datum}.
\end{definition}

By the expansion principle (Proposition~\ref{prop:expansion}),
realizability of a reduced datum implies realizability of the original datum.

\par\medskip
\noindent\textbf{The two reduction moves.}\par\nopagebreak
\smallskip
For brevity, we call a pair \((\lambda,\nu)\) of partitions of \(d\)
\emph{A1-small} if \(b(\lambda)+b(\nu)\le d-1\).

Move A1 (Proposition~\ref{prop:A1}) gives an output \(M\) with
\(b(M)=b(\lambda)+b(\nu)\) for every A1-small pair. It preserves the
candidate-datum condition and the source genus. Move A2
(Proposition~\ref{prop:A2}) applies to a candidate datum
\((\lambda,\nu,\mu_3,\ldots,\mu_k)\) when
\(b(\lambda)+b(\nu)\ge d-1\) and \(\sum_{i=3}^k b(\mu_i)\ge d-1\).
It produces a reduced candidate datum with \(\ell(M)\in\{1,2\}\), and
need not preserve the source genus. The main issue below is to choose
outputs satisfying \(g(M)=1\) while preserving the lcm condition.

\par\medskip
\noindent\textbf{Published realizability input.}\par\nopagebreak
\smallskip
Corollary~\ref{cor:special-reduced-data} gives the required realizability
input: a reduced candidate datum \((M,\mu_2,\ldots,\mu_r)\) with
\(r\ge3\), \(g(\mu_i)=1\) for \(2\le i\le r\), and
\(\ell(M)\in\{1,2\}\) is realizable.

\par\medskip
\noindent\textbf{Proof strategy.}\par\nopagebreak
\smallskip
Fix the degree \(d\). We use strong induction on the number \(k\) of
branches, uniformly over all nonnegative source genera, with the assumed
three-branch-point case as the base step. By the expansion principle, it suffices
to find a reduction whose outcome is either known to be realizable or is
Zieve-admissible with fewer branches. For the induction step, write
\(\Lambda=(\lambda_1,\ldots,\lambda_k)\), with \(k\ge4\), and first ask
whether there is an A1-small pair \((\lambda,\nu)\), that is,
\(b(\lambda)+b(\nu)\le d-1\).

\medskip
\noindent\textbf{Case 1: no A1-small pair exists.}
Then \(b(\lambda_i)+b(\lambda_j)\ge d\) for all \(1\le i<j\le k\).
Since \(k\ge4\), Move A2 applies to any selected pair: after reindexing,
the inherited branches also satisfy \(\sum_{i=3}^k b(\lambda_i)\ge d\).
Its output \(M\) satisfies \(\ell(M)\in\{1,2\}\). The inherited branches
satisfy \(g(\lambda_i)=1\), so
Corollary~\ref{cor:special-reduced-data} realizes the reduced datum.

\medskip
\noindent\textbf{Case 2: an A1-small pair exists.}
If some pair satisfies \(b(\lambda)+b(\nu)\in\{d-2,d-1\}\), its Move A1
output has \(\ell(M)=d-b(\lambda)-b(\nu)\in\{1,2\}\), giving the same
known realizability conclusion. We may therefore assume that every
A1-small pair satisfies \(b(\lambda)+b(\nu)\le d-3\). Move A1 already
preserves the candidate-datum equation and the source genus; the remaining
task is to arrange \(g(M)=1\) and the lcm condition for the reduced datum.

We call a branch partition \emph{fixed-point} if it contains a part \(1\),
and \emph{fixed-point-free} otherwise. If \(B,B'\) are fixed-point-free
with \(g(B)=g(B')=1\), Lemma~\ref{lem:fpf-defect} gives
\(b(B)+b(B')\ge2\lfloor d/2\rfloor+2>d-1\). Thus they cannot form an
A1-small pair, leaving two possible fixed-point patterns.

\smallskip
\noindent\textbf{Subcase 2a: an A1-small pair of fixed-point branches exists.}
Section~\ref{sec:fixed-fixed} constructs permutation products using
bipartite incidence forests. The edge counts of their tree components are
the parts of the output \(M\), allowing us to control \(g(M)\) and
\(m(M)\).

The reduction pair must also be chosen appropriately. If a branch \(A\)
satisfies \(m(A)=2\), then \(g(A)=1\) forces it to have a fixed point and
gives \(b(A)\le\lfloor(d-1)/2\rfloor\). Put
\(q=\#\{i:m(\lambda_i)=2\}\). This bound allows us to adjust the pair
while preserving the A1-small condition.
If \(q\le1\), we can choose \((\lambda,\nu)\) so that the inherited
branches \(\mu_3,\ldots,\mu_k\) satisfy \(m(\mu_i)\ne2\), and at least
one input partition has a part at least \(3\). We construct tree
components of weight \(1\) and of weight at least \(4\), giving
\(g(M)=1\) and \(m(M)\ge4\). These properties ensure the lcm condition
for the reduced datum. If \(q\ge2\), we choose
\(m(\lambda)=m(\nu)=2\). In this case, path and star forests provide
outputs with \(g(M)=1\), and \(m(M)\) can be chosen to take a value
compatible with the lcm condition of the reduced datum, as determined by
the classification in Lemma~\ref{lem:lcm-obstructions}.

\smallskip
\noindent\textbf{Subcase 2b: the mixed residual case.}
After the preceding cases, we may restrict to the mixed residual data of
Proposition~\ref{prop:standard-residual}. Every A1-small pair is then of
the form \((F,B)\), where \(F\) has fixed points, \(B\) is fixed-point-free,
and \(b(F)+b(B)\le d-3\). Such a pair is called a \emph{mixed A1-small
pair}. We distinguish the two possibilities for the source genus \(h\).

If \(h>0\), Section~\ref{sec:residual-positive} constructs a permutation
product with output \(M\) satisfying
\[
  \Cl_M\subset\Cl_F\Cl_B,\qquad
  g(M)=1,\qquad b(M)=b(F)+b(B)-2.
\]
The construction produces a fixed point and lowers the total defect by
two, so the source genus becomes \(h-1\). By choosing the cycles used in
the construction, we can also arrange for \(m(M)\) to take a value
compatible with the lcm condition of the reduced datum. The result is
Zieve-admissible with one fewer branch, so the induction hypothesis applies.

If \(h=0\), the genus-lowering construction is unavailable.
Section~\ref{sec:residual-zero} first uses defect bounds to show that a
mixed residual datum has the form \(\Lambda=(F,B_1,B_2,B_3)\), with \(F\)
having fixed points and each \(B_i\) fixed-point-free. A forest construction
realizes Move A1, and a suitable choice of the mixed pair and of the
cycles to be merged yields an output \(M\) with \(g(M)=1\).
Let \(B_j,B_k\) be the two inherited branches. Since they are
fixed-point-free and satisfy \(g(B_j)=g(B_k)=1\), we have
\(m(B_j),m(B_k)\ge6\); since \(M\) is nontrivial, \(m(M)\ge2\). Hence
\[
  \frac1{m(M)}+\frac1{m(B_j)}+\frac1{m(B_k)}
  \le\frac12+\frac16+\frac16=\frac56<1.
\]
The reduced triple therefore satisfies the lcm condition and is
Zieve-admissible. The assumed three-branch-point case and expansion realize the
original datum, completing the induction.

The paper is organized as follows.  Section~\ref{sec:preliminaries} records the
reduction moves, the expansion principle, the published realizability results
for reduced data whose new branch is \([d]\), \([d-1,1]\), or has two parts,
the corresponding exceptional cases, and the incidence-forest formalism.
Section~\ref{sec:fixed-fixed} constructs suitable Move A1 outputs when both
selected branches have fixed points.  Sections~\ref{sec:residual-positive} and
\ref{sec:residual-zero} treat mixed residual data of positive and zero source
genus, respectively.  The final section assembles these reductions into the
proof of Theorem~\ref{thm:main}.

\section{Preliminaries and reduction tools}
\label{sec:preliminaries}

This section collects the notation and reduction tools used throughout the
proof.  We first record elementary bounds for branch partitions and then state
the expansion principle and Moves A1 and A2 in the precise forms needed below.
We also classify failures of the lcm condition and combine the published
realizability results for a new branch \([d]\) or one with two parts,
excluding their exceptional families using the gcd-one inherited branches.  The final subsection develops the incidence-forest
formalism used to construct products realizing Move A1.

\subsection{Elementary bounds and the lcm condition}
\label{subsec:branch-data}

We use throughout the notation
\(\ell(\lambda)\), \(b(\lambda)\), \(g(\lambda)\), \(m(\lambda)\), and
\(\Cl_\lambda\) introduced in the Introduction.

We record two elementary bounds for later use.

\begin{lemma}[Fixed-point-free branch bounds]\label{lem:fpf-bounds}
\label{lem:fpf-lcm}\label{lem:fpf-defect}
If \(B\vdash d\) is fixed-point-free and gcd-one, then
\[
  m(B)\ge6,\qquad
  b(B)\ge \left\lfloor\frac d2\right\rfloor+1.
\]
\end{lemma}

\begin{proof}
If \(m(B)=2\), all parts are \(2\), so \(g(B)=2\).  If \(m(B)=3\), all parts
are \(3\), so \(g(B)=3\).  If \(m(B)=4\), all parts are \(2\) or \(4\), so
\(g(B)\ge2\).  If \(m(B)=5\), all parts are \(5\), so \(g(B)=5\).  Therefore
no fixed-point-free gcd-one partition has lcm \(<6\).

Since \(B\) is fixed-point-free, every part is at least \(2\), so
\(\ell(B)\le d/2\).  If \(d\) is odd, this gives
\(\ell(B)\le(d-1)/2\).  If \(d\) is even, equality
\(\ell(B)=d/2\) would force \(B=[2^{d/2}]\), whose gcd is \(2\).  Hence in the
even case \(\ell(B)\le d/2-1\).  Both cases give the stated defect bound.
\end{proof}

\begin{lemma}[Failures of the lcm condition]\label{lem:lcm-obstructions}
Let \(m_1,\ldots,m_r\ge2\) be integers, with \(r\ge3\). The equality
\[
  \sum_{i=1}^r\left(1-\frac1{m_i}\right)=2
\]
is impossible for \(r\ge5\). For \(r=4\) it holds exactly when every
\(m_i=2\). For \(r=3\), its solutions, up to order, are
\[
  (3,3,3),\qquad (2,4,4),\qquad (2,3,6).
\]
In particular, for fixed \(u,v\ge2\), the equation
\(1/t+1/u+1/v=1\) has at most one integer solution \(t\ge2\).
If \(u,v\ne2\), such a solution belongs to \(\{2,3\}\).
\end{lemma}

\begin{proof}
Each summand is at least \(1/2\), which proves the assertions for
\(r\ge4\). For \(r=3\), order the entries increasingly. The reciprocal
sum is \(1\), so the first entry is \(2\) or \(3\). If it is \(3\), all
three entries are \(3\). If it is \(2\), the second entry is \(3\) or
\(4\), giving the two remaining solutions. The last assertions follow
from this list, or from the uniqueness of a positive solution to
\(1/t=1-1/u-1/v\) for fixed \(u,v\).
\end{proof}

We use the notions of candidate datum, source genus, Zieve-admissibility, and
realizability introduced before Theorem~\ref{thm:main}.  For later use, if
\(D=(\lambda_1,\ldots,\lambda_k)\), set
\[
  v(D)=\sum_{i=1}^k b(\lambda_i).
\]
Thus a datum \(D\) whose entries are nontrivial is a candidate
datum of source genus \(h\) precisely when \(v(D)=2(d+h-1)\).
In particular, following the convention of Baroni and
Petronio \cite{BaroniPetronio2024}, every branch partition in a candidate datum
is nontrivial.

\subsection{Reduction moves and expansion}
\label{subsec:reduction-tools}

We record the expansion principle and the algebraic reduction moves A1 and A2
in the forms used below, with references to the original sources.

We use the term \emph{A1-small} as defined in the
Introduction.

\begin{proposition}[Expansion]\label{prop:expansion}
Let \(\Lambda'\) be obtained from
\(\Lambda=(\lambda_1,\ldots,\lambda_k)\) by replacing two entries
\(\lambda_i,\lambda_j\) by a partition \(\mu\), and suppose
\[
  \Cl_\mu\subset \Cl_{\lambda_i}\Cl_{\lambda_j}.
\]
If \(\Lambda'\) is realizable, then \(\Lambda\) is realizable.
\end{proposition}

\begin{proof}
Take a transitive permutation realization of \(\Lambda'\), and let
\(\tau\in\Cl_\mu\) be the factor corresponding to the branch created by
the replacement. By the class-product inclusion, write
\[
  \tau=\alpha\beta,\qquad
  \alpha\in\Cl_{\lambda_i},\quad \beta\in\Cl_{\lambda_j}.
\]
Replace \(\tau\) by the adjacent factors \(\alpha,\beta\). The product
remains the identity, and the resulting sequence of cycle types is an
indexed representative of \(\Lambda\). The new generated group contains
every generator of the original transitive group, since
\(\tau=\alpha\beta\), so it is transitive. This realizes \(\Lambda\).
\end{proof}

The existence assertion in the next proposition is the defect-preserving
product statement of Edmonds, Kulkarni, and Stong \cite[Lemma~4.2]{EKS1984}, also
recorded by Baroni and Petronio
\cite[Proposition~2.2]{BaroniPetronio2024}.  Baroni and Petronio package the
resulting reduction as Move A1
\cite[Proposition~2.5]{BaroniPetronio2024}.

\begin{proposition}[Baroni--Petronio Move A1]\label{prop:A1}
If \(\lambda,\nu\vdash d\) satisfy
\[
  b(\lambda)+b(\nu)\le d-1,
\]
then there exists a partition \(\mu\vdash d\) such that
\[
  \Cl_\mu\subset \Cl_\lambda\Cl_\nu,
  \qquad
  b(\mu)=b(\lambda)+b(\nu).
\]
\end{proposition}

For a candidate datum \(D=(\lambda,\nu,\mu_3,\ldots,\mu_k)\)
satisfying this bound, replacing \(\lambda,\nu\) by an output \(\mu\)
of Proposition~\ref{prop:A1} produces \(D'=(\mu,\mu_3,\ldots,\mu_k)\);
this is Move A1. The new branch is nontrivial, and defect additivity gives
\(v(D')=v(D)\), so \(D'\) is again a candidate datum of the same source genus.

The next proposition is Move A2 of Baroni and Petronio
\cite[Proposition~2.6]{BaroniPetronio2024}, based on the product
constructions in \cite[Propositions~2.3 and~2.4]{BaroniPetronio2024}.

\begin{proposition}[Baroni--Petronio Move A2]\label{prop:A2}
Let
\[
  D=(\lambda,\nu,\mu_3,\ldots,\mu_k)
\]
be a candidate datum of degree \(d\ge3\), and suppose that
\[
  b(\lambda)+b(\nu)\ge d-1
  \qquad\text{and}\qquad
  \sum_{i=3}^k b(\mu_i)\ge d-1.
\]
Then there is a partition \(\pi\vdash d\) such that
\[
  \Cl_\pi\subset\Cl_\lambda\Cl_\nu
\]
and the reduced datum
\[
  D'=(\pi,\mu_3,\ldots,\mu_k)
\]
is again a candidate datum.  More precisely,
\[
  \pi=
  \begin{cases}
    [d/2,d/2],
      & \lambda=\nu=[2^{d/2}],\\
    [d],
      & b(\lambda)+b(\nu)\equiv d-1\pmod 2,\\
    [d-1,1],
      & b(\lambda)+b(\nu)\equiv d\pmod 2
        \text{ and }(\lambda,\nu)\ne([2^{d/2}],[2^{d/2}]).
  \end{cases}
\]
The first case occurs only when \(d\) is even.
\end{proposition}

For later use, we record the following defect consequence of
this statement.  Since
\(b(\pi)\ge d-2\) and the branches \(\mu_3,\ldots,\mu_k\) inherited from
\(D\) have total defect at least \(d-1\), the total defect of \(D'\) is at
least \(2d-3\).  It is even because \(D'\) is a candidate datum, and hence it
is at least \(2d-2\).
\subsection{Realizability with a branch partition of length at most two}
\label{subsec:special-output-tools}

We first state the published realizability results in the candidate-datum
setting, and then derive the consequence for gcd-one inherited branches.

\begin{proposition}[Published realizability results]
\label{prop:published-special-reduced-data}
Let \(D\) be a degree-\(d\) candidate datum with \(r\ge3\) branches,
and let \(M\) be a specified branch of \(D\). Then the following hold.
\begin{enumerate}[label=\textnormal{(\alph*)}]
\item If \(M=[d]\), then \(D\) is realizable
\cite[Proposition~5.2]{EKS1984}.
\item If \(\ell(M)=2\), then \(D\) is realizable if and only if it is
not one of the thirteen exceptional cases listed in Baroni and Petronio
\cite[Theorem~3.1]{BaroniPetronio2024}.
\end{enumerate}
The second assertion includes \(M=[d-1,1]\).
\end{proposition}

\begin{corollary}[Gcd-one inherited branches]
\label{cor:special-reduced-data}
\label{prop:special-reduced-data}
Let \(D=(M,\mu_2,\ldots,\mu_r)\) be a degree-\(d\) candidate datum
with \(r\ge3\), and suppose \(g(\mu_i)=1\) for \(2\le i\le r\).
If \(M=[d]\) or \(\ell(M)=2\), then \(D\) is realizable.
\end{corollary}

\begin{proof}
If \(M=[d]\), apply Proposition~\ref{prop:published-special-reduced-data}(a).
If \(\ell(M)=2\), apply part~(b) of the same proposition. We group the
thirteen exceptional cases in Table~\ref{tab:length-two-exceptions},
retaining their numbering in \cite[Theorem~3.1]{BaroniPetronio2024}.
At most one branch, namely \(M\), is allowed to have gcd greater than one.

\begin{table}[H]
\centering
\small
\renewcommand{\arraystretch}{1.2}
\begin{tabular}{@{}c p{0.72\linewidth}@{}}
\hline
Items & Obstruction to gcd-one inherited branches \\
\hline
\((1),(4),(5),(8)\)--\((12)\) &
There is a branch \([2^N]\) with \(N\ge3\). It cannot be the
length-two branch \(M\), so it is inherited and has gcd \(2\). \\
\((2),(7),(13)\) &
At least two branches have gcd greater than one, so at least one
of them is inherited. \\
\((3),(6)\) &
For \(k>2\), the first branch, \([2^k]\) or \([h^k]\) with \(h\ge2\),
has more than two parts and cannot be \(M\). For \(k=2\), there
are two branches equal to \([2,2]\) or two equal to \([h,h]\),
respectively, so a branch of gcd greater than one is inherited. \\
\hline
\end{tabular}
\caption{Obstructions to gcd-one inherited branches in the thirteen
exceptional cases.}
\label{tab:length-two-exceptions}
\end{table}

Every listed case therefore contradicts \(g(\mu_i)=1\), so the
corollary follows. No gcd-one condition is required of \(M\).
\end{proof}

\subsection{Incidence forests for Move A1}
\label{subsec:incidence-forests}

The encoding of a bipartite map by permutations acting on its
edge set is standard in the theory of dessins and hypermaps; see Lando and
Zvonkin \cite[Section~1.5.1 and
Proposition~1.5.3]{LandoZvonkin2004}.
We record the special forest case in the form used throughout this paper.

\begin{definition}[Bipartite incidence multigraph]
\label{def:incidence-graph}
Let \(\alpha,\beta\in\operatorname{Sym}(\Omega)\), where \(\Omega\)
is a finite set. The bipartite incidence multigraph
\(\Gamma(\alpha,\beta)\) has one black vertex for each cycle of
\(\alpha\) and one white vertex for each cycle of \(\beta\), including
cycles of length one. For each \(x\in\Omega\), an edge labelled \(x\)
joins the black and white vertices corresponding to the cycles
containing \(x\).

Thus the degree of a vertex is the length of its corresponding cycle.
For a subgraph \(K\), we write
\[
  w(K)=|E(K)|
\]
for its weight.
\end{definition}

Conversely, let \(\Gamma\) be a bipartite multigraph without isolated
vertices, with its edges labelled by \(\Omega\). Choose a cyclic order
of the incident edges at each vertex. Reading these orders at the
black vertices defines a permutation \(\alpha\), and reading them at
the white vertices defines a permutation \(\beta\). Then
\(\Gamma=\Gamma(\alpha,\beta)\), and the cycle types of \(\alpha\)
and \(\beta\) are the black and white vertex-degree partitions,
respectively.

The sets of edge labels in the connected components of
\(\Gamma(\alpha,\beta)\) are precisely the orbits of
\(\langle\alpha,\beta\rangle\) on \(\Omega\). Indeed, moving between
edges incident to a single vertex corresponds to applying a power
of \(\alpha\) or \(\beta\).

For \(\gamma\in\operatorname{Sym}(\Omega)\), let \(\ell(\gamma)\)
be its number of cycles, including fixed points, and put
\(b(\gamma)=|\Omega|-\ell(\gamma)\).

\begin{proposition}[Incidence-forest criterion]
\label{prop:incidence-forest}
Let \(|\Omega|=d\) and let \(\alpha,\beta\in\operatorname{Sym}(\Omega)\).
Then
\[
  b(\alpha\beta)\le b(\alpha)+b(\beta).
\]
Equality holds if and only if \(\Gamma(\alpha,\beta)\) is a forest.  In the
equality case, every connected component \(K\) supports exactly one cycle of
\(\alpha\beta\), and this cycle has length \(w(K)\).  Consequently, the
cycle type of \(\alpha\beta\) is the partition formed by the component
weights of \(\Gamma(\alpha,\beta)\).
\end{proposition}

\begin{proof}
Put
\[
  V=\ell(\alpha)+\ell(\beta),\qquad E=d,
\]
and let \(c\) be the number of connected components of \(\Gamma\).  Every
cycle of \(\alpha\beta\) lies in one orbit of
\(\langle\alpha,\beta\rangle\), hence
\[
  \ell(\alpha\beta)\ge c.
\]
For any finite multigraph, \(c\ge V-E\).  Therefore
\[
\begin{aligned}
  b(\alpha\beta)
    &=d-\ell(\alpha\beta)\\
    &\le d-(V-E)\\
    &=(d-\ell(\alpha))+(d-\ell(\beta))
      =b(\alpha)+b(\beta).
\end{aligned}
\]

Suppose equality holds.  Then
\[
  V-E=\ell(\alpha\beta)\ge c\ge V-E,
\]
so \(c=V-E\).  This is the Euler criterion for every component of
\(\Gamma\) to be a tree.  The same chain of equalities gives
\(\ell(\alpha\beta)=c\), so there is exactly one product cycle in each
component.

Conversely, suppose \(\Gamma\) is a forest.  We prove by induction on the
number of edges in a tree component that \(\alpha\beta\) has one cycle on the
edge labels of that component.
For a one-edge component this is immediate.  For a larger tree, choose a leaf
edge \(e_x\).  If its leaf is black, then \(x\) is a fixed point of
\(\alpha\); deleting \(e_x\) deletes \(x\) from its cycle of \(\beta\).
By induction, the product on the remaining edges is one cycle, and restoring
\(x\) inserts it into that cycle.  The case of a white leaf is symmetric.
Thus each component \(K\) gives one cycle containing all \(w(K)\) of its edge
labels.  Hence \(\ell(\alpha\beta)=c=V-E\), which gives equality of the
defects and proves the final assertion.
\end{proof}

We now use this correspondence to complete prescribed tree components
to products realizing Move A1.

\begin{corollary}[Forest realization and completion]
\label{cor:forest-completion}
Let \(\lambda,\nu\vdash d\) satisfy
\[
  b(\lambda)+b(\nu)\le d-1.
\]
Then there exist \(\alpha\in\Cl_\lambda\) and
\(\beta\in\Cl_\nu\) such that \(\Gamma(\alpha,\beta)\) is a forest.  Its
component weights form the cycle type of a product realizing Move A1.

More generally, suppose that some disjoint incidence-tree components have
already been prescribed using whole parts of \(\lambda\) and \(\nu\), and
that the unused parts form partitions \(\lambda'\) and \(\nu'\) of a common
remaining degree \(d'>0\).  If
\[
  b(\lambda')+b(\nu')\le d'-1,
\]
then the prescribed components can be completed by a disjoint incidence
forest on the unused symbols.  The union is an incidence forest for a product
of cycle types \(\lambda\) and \(\nu\) that realizes Move A1.
\end{corollary}

\begin{proof}
Proposition~\ref{prop:A1} gives a partition \(\mu\vdash d\) with
\[
  \Cl_\mu\subset\Cl_\lambda\Cl_\nu,
  \qquad b(\mu)=b(\lambda)+b(\nu).
\]
Choose \(\alpha\in\Cl_\lambda\) and \(\beta\in\Cl_\nu\) whose product has
cycle type \(\mu\).  Proposition~\ref{prop:incidence-forest} shows that
\(\Gamma(\alpha,\beta)\) is a forest and identifies its component weights
with the parts of \(\mu\).

For the completion statement, apply the first part to
\(\lambda',\nu'\) on an alphabet disjoint from the labels of the prescribed
components.  Taking the disjoint union of the resulting forest with the
prescribed trees defines permutations of cycle types \(\lambda,\nu\).  Their
incidence graph is a forest, so Proposition~\ref{prop:incidence-forest}
again shows that the defect of their product is the sum of their defects.
\end{proof}

For the constructive arguments below, we also need to recognize which
collections of prescribed cycle lengths can form a tree component.

\begin{lemma}[Balanced tree-component criterion]
\label{lem:balanced-tree-component}
Let \(P\) and \(Q\) be nonempty finite sets of prospective black and white
vertices, and assign a positive integer prescribed degree
\(a_v\) to every \(v\in P\sqcup Q\).  There is a bipartite tree on
\(P\sqcup Q\) whose vertex degrees are the prescribed degrees
if and only if, for some positive integer
\(S\),
\[
  \sum_{v\in P}a_v=\sum_{v\in Q}a_v=S,
  \qquad |P|+|Q|=S+1.
\]
Such a tree is an incidence component of weight \(S\).
\end{lemma}

\begin{proof}
Necessity follows by counting edges on the two colour classes and using
\(|E|=|V|-1\) for a tree.  For sufficiency, argue by induction on
\(|P|+|Q|\).  If one colour class has one vertex, the equations force a star.
Otherwise the equality \(S=|P|+|Q|-1\) implies that some prescribed degree is
one.  Choose such a vertex \(v\).  The opposite colour class
contains a vertex \(u\) whose prescribed degree is greater than one: otherwise
all degrees in that class would be one, and its degree sum would force the
colour class containing \(v\) to have only one vertex.  Remove \(v\) and decrease the degree
of \(u\) by one.  The two displayed equalities are preserved for the smaller
degree assignment, so induction supplies a tree; attaching \(v\) to \(u\) completes the
required tree.
\end{proof}

More generally, partition all parts of \(\lambda\) and \(\nu\) into paired
blocks \((P_j,Q_j)\) satisfying Lemma~\ref{lem:balanced-tree-component}
with common weights \(S_j\). Choose a tree for each block and take their
disjoint union. After labelling its edges and choosing cyclic orders at
the vertices, Proposition~\ref{prop:incidence-forest} gives permutations
of cycle types \(\lambda,\nu\) whose product has cycle type
\([S_1,\ldots,S_r]\), and hence realizes Move A1.

\section{Reduction of two fixed-point branches}
\label{sec:fixed-fixed}

We prove the fixed-point reduction needed to reach the mixed residual
case. The pair to be reduced may be chosen after the branch lcms are known.
We combine a general construction avoiding a prescribed lcm in
\(\{2,3\}\) with special constructions for two lcm-\(2\) branches.
The pair-selection argument below shows that these suffice once the
boundary cases have been excluded. The possible lcm failures are listed in
Lemma~\ref{lem:lcm-obstructions}.

\subsection{Forest constructions and avoidance of small lcms}
\label{subsec:forest-flexibility}

Let \(\lambda,\nu\vdash d\) be nontrivial partitions with fixed points,
written as
\[
  \lambda=[c_1,\ldots,c_a,1^{x+1}],\qquad
  \nu=[e_1,\ldots,e_b,1^{y+1}],
\]
where \(a,b\ge1\), \(c_i,e_j\ge2\), and \(x,y\ge0\).
We construct bipartite forests whose black and white vertex degrees
are the parts of \(\lambda\) and \(\nu\), respectively.

Reserve one degree-one vertex of each colour and join them by an
edge. This gives an isolated component, which contributes a part
\(1\) to the product partition. There remain \(x\) black and \(y\)
white degree-one vertices, together with the vertices of degrees
\(c_1,\ldots,c_a\) and \(e_1,\ldots,e_b\).

Put
\[
  C=\sum_{i=1}^a c_i,\qquad
  \delta_j=e_j-1,\qquad
  D=b(\lambda)+b(\nu).
\]
The identities
\[
  d=C+x+1=\sum_{j=1}^b e_j+y+1,\qquad
  b(\lambda)=C-a,\qquad
  b(\nu)=\sum_{j=1}^b\delta_j
\]
show that \(D\le d-3\) implies
\begin{equation}\label{eq:fixed-fixed-strict-bounds}
  \sum_{j=1}^b\delta_j\le x+a-2,\qquad
  C-a\le y+b-2.
\end{equation}

To construct an additional tree component, choose index sets
\[
  I\subseteq\{1,\ldots,a\},\qquad
  J\subseteq\{1,\ldots,b\},
\]
not both empty, and put
\[
  C_I=\sum_{i\in I}c_i,\qquad
  E_J=\sum_{j\in J}e_j,\qquad
  n=|I|+|J|.
\]
We seek a tree on the selected vertices, together with \(p\) of
the remaining black degree-one vertices and \(q\) of the remaining
white degree-one vertices.

If this tree has \(S\) edges, counting degrees on each colour class
and counting vertices gives
\[
  C_I+p=E_J+q=S,\qquad
  n+p+q=S+1.
\]
Solving these equations yields
\begin{equation}\label{eq:component-parameters}
  p=E_J-n+1,\qquad
  q=C_I-n+1,\qquad
  S=C_I+E_J-n+1.
\end{equation}

Conversely, suppose the values in
\eqref{eq:component-parameters} satisfy
\[
  0\le p\le x,\qquad 0\le q\le y,\qquad S>0.
\]
The required degree-one vertices are then available. The prescribed
degrees have sum \(S\) on each colour class, and there are \(S+1\)
vertices in total. Lemma~\ref{lem:balanced-tree-component} therefore
provides the required tree component.

Under the hypothesis \(D\le d-3\), the reserved isolated edge can be
completed to a forest by Corollary~\ref{cor:forest-completion}.
The same holds after prescribing any additional tree component
constructed above. Such a component with \(S\) edges contributes
total defect \(S-1\). The unused parts therefore have common degree
\(d'=d-1-S\) and total defect
\[
  D'=D-(S-1)\le d'-1.
\]
Moreover, \(S-1\le D\le d-3\) implies \(d'>0\).
Corollary~\ref{cor:forest-completion} therefore applies to the unused
parts. By Proposition~\ref{prop:incidence-forest}, the resulting product
partition contains a part \(1\) from the isolated edge and a part
\(S\) from the additional component.

\begin{lemma}[Two branches of lcm \(2\)]\label{lem:involution-outputs}
Let
\[
  \lambda=[2^a,1^{d-2a}],\qquad
  \nu=[2^b,1^{d-2b}],\qquad a\ge b\ge1,
\]
where both fixed-point multiplicities are positive and \(a+b\le d-3\).
For every prescribed \(t\in\{2,3,4,6\}\), there is a Move A1 output
\(M\) with \(g(M)=1\) and \(m(M)\ne t\).
\end{lemma}

\begin{proof}
The outputs in Table~\ref{tab:involution-outputs} suffice. In the first two
rows use the displayed output unless its lcm equals the prescribed value;
the last three rows give replacements for precisely those cases.

\begin{table}[H]
\centering
\small
\renewcommand{\arraystretch}{1.3}
\begin{tabular}{@{}ccc@{}}
\hline
Condition & Output \(M\) & \(m(M)\) \\
\hline
\(a=b\) & \([2a+1,1^{d-2a-1}]\) & \(2a+1\) \\
\(a>b\) & \([2b+2,2^{a-b-1},1^{d-2a}]\) & \(2b+2\) \\
\(a=b=1,\ t=3\) & \([2,2,1^{d-4}]\) & \(2\) \\
\(a>b=1,\ t=4\) & \([3,2^{a-1},1^{d-2a-1}]\) & \(6\) \\
\(a>b=2,\ t=6\) & \([3,4,2^{a-3},1^{d-2a-1}]\) & \(12\) \\
\hline
\end{tabular}
\caption{Move A1 outputs for two involution types.}
\label{tab:involution-outputs}
\par\smallskip
{\footnotesize Parts with multiplicity zero are omitted.\par}
\end{table}

We describe forests producing all five rows. A path containing \(u\)
black and \(u\) white degree-two vertices has one degree-one vertex of each
colour and weight \(2u+1\). A path containing \(u+1\) black and \(u\)
white degree-two vertices has two white degree-one vertices and weight
\(2u+2\). A single black degree-two vertex with two white degree-one vertices
is a star of weight \(2\).

For the first row, put all degree-two vertices into the first kind of
path. The unused degree-one vertices pair into isolated edges. For the second
row, use the second kind of path with \(u=b\), and turn the remaining
\(a-b-1\) black degree-two vertices into stars. This uses \(2(a-b)\)
white degree-one vertices and no black degree-one vertices; exactly \(d-2a\)
degree-one vertices remain on each side.

For the third row, use two disjoint stars, one centred at each colour;
\(d\ge5\) ensures that their leaves are available. For the fourth row,
use a path of weight \(3\), containing one degree-two vertex of each
colour and one degree-one vertex of each colour, followed by \(a-1\) black
stars. This uses one black and \(2a-1\) white degree-one vertices. For the
fifth row, use disjoint paths of weights \(3\) and \(4\), with one
black and one white degree-two vertex in the first and two black and
one white degree-two vertex in the second. Add \(a-3\) black stars.
This uses one black and \(2a-3\) white degree-one vertices. In the last two
rows the unused degree-one vertices pair into \(d-2a-1\) isolated edges;
their availability follows from \(d-2a\ge1\). Thus every row is an
incidence forest using exactly the prescribed vertices.

Each displayed output has defect \(a+b\). In the first two rows it
has a part \(1\): for \(a=b\), the bound \(2a\le d-3\) gives
\(d-2a-1\ge2\). The third row also has a part \(1\). The fourth
and fifth rows have coprime parts \(3,2\) and \(3,4\), respectively.
Hence all five outputs are gcd-one. The first-row lcm belongs to
\(\{2,3,4,6\}\) only when \(a=b=1\), giving \(3\); the second-row
lcm belongs to this set only when \(b=1\) or \(2\), giving \(4\)
or \(6\). The replacements avoid the corresponding value.
\end{proof}

\begin{lemma}[Avoiding lcm \(2\) or \(3\)]\label{lem:fixed-fixed-flex}
Let \(\lambda,\nu\vdash d\) be nontrivial partitions with fixed points,
and suppose \(b(\lambda)+b(\nu)\le d-3\). There is a Move A1 output
containing a part \(1\). For each prescribed \(t\in\{2,3\}\), there
is a possibly different Move A1 output \(M\) with
\[
  g(M)=1,\qquad m(M)\ne t.
\]
If one of the input partitions has a part at least \(3\), one output
can be chosen to contain both a part \(1\) and a part at least \(4\).
\end{lemma}

\begin{proof}
The first assertion follows from Corollary~\ref{cor:forest-completion}
after reserving an isolated edge, as in the completion argument above.
If all nontrivial parts are
\(2\), apply Lemma~\ref{lem:involution-outputs}.
Otherwise use the fixed-edge setup and, after interchanging colours,
choose a black vertex of degree \(c\ge3\). We construct an additional
component of weight \(S\ge4\); completion then gives an output with
a part \(1\) and \(m(M)\ge S\ge4\).

\emph{Case 1: \(y<c\).}
Put \(k=c-y\). Since \(C-a\ge c+a-2\), the second inequality in
\eqref{eq:fixed-fixed-strict-bounds} implies \(b\ge a+k\). Use the
chosen black vertex and any \(k\) white nontrivial vertices, relabelled
as the first \(k\). Formula \eqref{eq:component-parameters} gives
\[
  p=\sum_{j=1}^k\delta_j,\qquad q=c-k=y,
  \qquad S=c+\sum_{j=1}^k\delta_j\ge4.
\]
There are at least \(a\) unused white nontrivial vertices. Consequently,
the first inequality in \eqref{eq:fixed-fixed-strict-bounds} gives
\[
  x-p\ge\sum_{j>k}\delta_j-a+2\ge2.
\]
Thus the required leaves are available.

\emph{Case 2: \(y\ge c\).}
If \(c\ge4\), use the black star of weight \(c\). It remains to treat
\(c=3\). If \(\delta_1\le x\), the component on the chosen black
vertex and the first white vertex has
\[
  (p,q,S)=(\delta_1,2,3+\delta_1),
\]
so it is available and has weight at least \(4\).

If \(\delta_1>x\), put \(h=\delta_1-x\). The first inequality in
\eqref{eq:fixed-fixed-strict-bounds} gives
\begin{equation}\label{eq:small-lcm-augmentation}
  h+\sum_{j>1}\delta_j\le a-2.
\end{equation}
Choose \(h\) other black nontrivial vertices; this is possible because
\(h\le a-2\). Let \(\Gamma\) be their total defect and \(\Gamma'\)
the total defect of the black nontrivial vertices left unused. The
component consisting of these \(h+1\) black vertices and the first
white vertex has
\[
  p=x,\qquad q=\Gamma+2,\qquad S=3+\delta_1+\Gamma\ge4.
\]
By \eqref{eq:small-lcm-augmentation} and
\(\sum_{j>1}\delta_j\ge b-1\), there are at least \(b\) unused
black nontrivial vertices, so \(\Gamma'\ge b\). The other inequality
in \eqref{eq:fixed-fixed-strict-bounds} now yields
\[
  2+\Gamma+\Gamma'\le y+b-2,
  \qquad y\ge\Gamma+4>q.
\]
The component is therefore available. In every case, the completion
argument above, based on
Corollary~\ref{cor:forest-completion}, extends this component and the
reserved isolated edge to the required forest. This proves the last
assertion and the avoidance of both small lcm values.
\end{proof}

\subsection{Choosing the reduction pair and identifying mixed residuals}
\label{subsec:fixed-fixed-reduction}

\begin{proposition}[A fixed-point reduction after the boundary cases]
\label{prop:fixed-fixed}
Let \(\Lambda\) be a Zieve-admissible datum with \(k\ge4\) branches.
Assume that every A1-small pair has defect sum at most \(d-3\), and
that some A1-small pair consists of two fixed-point branches. Then some
pair of fixed-point branches admits a Move A1 reduction to a
Zieve-admissible datum with \(k-1\) branches and the same source genus.
\end{proposition}

\begin{proof}
For any pair considered below, a gcd-one output of Move A1 has positive
defect and preserves the candidate-datum equation. The inherited
branches retain their nontriviality and gcd-one properties. It is
therefore enough to arrange the lcm condition.

If \(k\ge6\), use any fixed-point A1-small pair and the output with
a part \(1\) from Lemma~\ref{lem:fixed-fixed-flex}; the lcm condition is
automatic by Lemma~\ref{lem:lcm-obstructions}. If \(k=5\), use the same
pair. The only possible failure in the reduced quadruple has all lcms
equal to \(2\). If the three inherited lcms are all \(2\), use
Lemma~\ref{lem:fixed-fixed-flex} with \(t=2\); otherwise any output with
a part \(1\) suffices.

Suppose now that \(k=4\), and let \(q\) be the number of branches
of lcm \(2\) in \(\Lambda\). Every such branch \(A\) has type
\([2^u,1^{d-2u}]\), because its gcd is one and it is nontrivial.
In particular it has a fixed point and
\begin{equation}\label{eq:lcm-two-defect}
  b(A)=u\le\frac{d-1}{2}.
\end{equation}

\emph{Case 1: \(q=0\).}
Choose any fixed-point A1-small pair. Neither inherited lcm is \(2\).
By Lemma~\ref{lem:lcm-obstructions}, at most one value
\(t\in\{2,3\}\) can cause the reduced triple to fail the lcm condition.
If such a value exists, apply Lemma~\ref{lem:fixed-fixed-flex} with
that \(t\); otherwise use its output with a part \(1\).

\emph{Case 2: \(q=1\).}
Let \(A\) be the unique branch of lcm \(2\). If a fixed-point
A1-small pair already contains \(A\), choose it. Otherwise take any
such pair \((\lambda,\nu)\) and let \(F\) be its member of smaller
defect. Then \(F\ne A\), and
\[
  b(F)\le\frac{d-3}{2},\qquad
  b(F)+b(A)\le d-2\le d-1.
\]
Thus \((F,A)\) is A1-small. The hypothesis excluding boundary pairs
strengthens its defect sum to at most \(d-3\). In either choice the
reduction consumes \(A\), so neither inherited lcm is \(2\).
The argument in Case 1 applies to the chosen pair.

\emph{Case 3: \(q\ge2\).}
Choose two branches \(A,A'\) of lcm \(2\). By
\eqref{eq:lcm-two-defect}, their defect sum is at most \(d-1\), so
they are A1-small and hence, by hypothesis, have defect sum at most
\(d-3\). By Lemma~\ref{lem:lcm-obstructions}, at most one value
\(t\in\{2,3,4,6\}\) can cause the reduced triple to fail the lcm condition.
If such a value exists, Lemma~\ref{lem:involution-outputs} avoids
\(t\) with a gcd-one output; otherwise either basic output in its
table suffices.
This proves the proposition.
\end{proof}

\begin{samepage}
\begin{proposition}[Standard reduction to mixed residuals]
\label{prop:standard-residual}
Let \(\Lambda\) be a Zieve-admissible datum of degree \(d\) with
\(k\ge4\) branches. Then at least one of the following holds:
\begin{enumerate}[label=\textnormal{(\roman*)}]
\item \(\Lambda\) admits a reduction to a candidate datum satisfying
the hypotheses of Corollary~\ref{cor:special-reduced-data};
\item Move A1 produces a Zieve-admissible datum with \(k-1\) branches;
\item An A1-small pair exists, and every A1-small pair is of the form
\((F,B)\),
where \(F\) has fixed points, \(B\) is fixed-point-free, and
\[
  b(F)+b(B)\le d-3.
\]
\end{enumerate}
In case \textnormal{(iii)}, we call \(\Lambda\) a \emph{mixed residual
datum} and each such pair a \emph{mixed A1-small pair}.
\end{proposition}
\end{samepage}

\begin{proof}
Suppose that neither \textnormal{(i)} nor \textnormal{(ii)} holds.
If no A1-small pair existed, every pair would have defect sum at least
\(d\). Choose two branches to reduce. Since \(k\ge4\), two other
branches remain and also have defect sum at least \(d\). Thus
Proposition~\ref{prop:A2} applies. Its output is \([d]\), \([d-1,1]\), or a
partition with two parts, giving \textnormal{(i)} by
Corollary~\ref{cor:special-reduced-data}, a contradiction.

An A1-small pair therefore exists. If one had defect sum \(d-1\) or
\(d-2\), its Move A1 output would have one or two parts, again
giving \textnormal{(i)}. Hence every A1-small pair has defect sum at
most \(d-3\). If one consisted of two fixed-point branches,
Proposition~\ref{prop:fixed-fixed} would give \textnormal{(ii)}. Finally, two
fixed-point-free gcd-one branches cannot be A1-small: by
Lemma~\ref{lem:fpf-defect}, each has defect at least
\(\lfloor d/2\rfloor+1\). Every remaining A1-small pair is therefore
mixed, proving \textnormal{(iii)}.
\end{proof}

\section{Positive-genus mixed residuals}\label{sec:residual-positive}

This section treats mixed residual data of positive source genus.  For a
fixed-point branch \(F\) and a fixed-point-free branch \(B\), we construct a
genus-lowering product.  We then show that the
construction can be chosen so that the new branch has gcd one and the
reduced datum continues to satisfy the lcm condition.  The
resulting reduction decreases
both the number of branches and the source genus by one, so the strong
induction hypothesis applies.

\subsection{A genus-lowering product construction}

\begin{lemma}[Two local multiplication rules]\label{lem:local-rules}
Let products be composed from right to left.
In both parts, \(r\ge0\).
\begin{enumerate}[label=(\alph*)]
\item Let \(\beta\) contain disjoint cycles \(C_0,\ldots,C_r\), and choose one
letter \(p_i\in C_i\).  If
\[
  \alpha=(p_0\,p_1\,\cdots\,p_r),
\]
then the product \(\alpha\beta\) replaces \(C_0,\ldots,C_r\) by one cycle of
length \(|C_0|+\cdots+|C_r|\), leaving all other cycles unchanged.
\item Let \(\beta\) contain disjoint cycles \(C_0,\ldots,C_r\), where
\[
  C_0=(u_0\,u_1\,u_2\,\cdots\,u_{x-1})
\]
with \(x=|C_0|\ge2\) and \(u_0,u_1\) adjacent in the
displayed cyclic order.  Choose one letter
\(p_i\in C_i\) for \(1\le i\le r\).  If
\[
  \alpha=(u_0\,p_1\,p_2\,\cdots\,p_r\,u_1),
\]
then \(\alpha\beta\) replaces \(C_0,\ldots,C_r\) by the fixed point \(u_0\)
and one cycle of length
\[
  |C_0|+\cdots+|C_r|-1,
\]
leaving all other cycles unchanged.
\end{enumerate}
\end{lemma}

\begin{proof}
If \(r=0\), part (a) is immediate because \(\alpha=(p_0)\) is the identity.
For part (b), \(\alpha=(u_0\,u_1)\), and on the letters of \(C_0\) we have
\[
  \alpha\beta=(u_0)(u_1\,u_2\,\cdots\,u_{x-1}).
\]
For \(x=2\), the second cycle is \((u_1)\).
All other cycles are unchanged. We now assume \(r\ge1\).

For (a), the orbit under \(\alpha\beta\) follows each \(C_i\) until the
predecessor of \(p_i\), then jumps to \(\alpha(p_i)=p_{i+1}\), with indices
taken modulo \(r+1\). Thus the chosen cycles are joined into one cycle,
and no other letter is affected by \(\alpha\).

For (b), \((\alpha\beta)(u_0)=\alpha(u_1)=u_0\), so \(u_0\) is fixed.  Starting
at \(u_1\), the orbit follows \(C_0\) through
\[
  u_1,u_2,\ldots,u_{x-1}
\]
and then jumps from the predecessor of \(u_0\) to \(\alpha(u_0)=p_1\).  It then
traverses \(C_1,\ldots,C_r\) in order, jumping at each marked letter's
predecessor, and returns to \(u_1\) after \(C_r\). This gives the claimed
remaining cycle.
\end{proof}

\begin{lemma}[Genus-lowering product]\label{lem:genus-lowering-product}
Let
\(F=[c_1,\ldots,c_s,1^f]\vdash d\),
where \(s\ge1\), \(f\ge1\), and \(c_u\ge2\) for
\(1\le u\le s\), and put \(a=b(F)=\sum_{u=1}^s(c_u-1)\).
Let \(B=[x_1,\ldots,x_n]\vdash d\) be fixed-point-free and gcd-one.  If
\[
  n\ge a,
\]
then there exists a partition \(M\vdash d\) such that
\[
  \Cl_M\subset \Cl_F\Cl_B,
  \qquad
  g(M)=1,
  \qquad
  b(M)=b(B)+a-2.
\]
\end{lemma}

\begin{proof}
Choose a permutation \(\beta\) of type \(B\).  In each of the cases below, we
select \(a\) cycles of \(\beta\), denoted \(D_1,\ldots,D_a\), and construct a
permutation \(\alpha\) of type \(F\).
The nontrivial cycles of \(\alpha\) will be denoted
\(\alpha_1,\ldots,\alpha_s\), with \(|\alpha_u|=c_u\); all remaining letters
are fixed by \(\alpha\).
Once \(\alpha_1,\ldots,\alpha_u\) have been chosen, write
\[
  \beta_u=\alpha_u\cdots\alpha_1\beta.
\]

\medskip
\noindent\textbf{Case 1: some \(c_u\ge3\).}
Choose any \(a\) cycles \(D_1,\ldots,D_a\) of \(\beta\).
Relabel so that \(c_1\ge3\), and define
\[
  r_u=(c_1-1)+\sum_{v=2}^u(c_v-1)\qquad(1\le u\le s).
\]
Thus \(r_1=c_1-1\) and \(r_s=a\).  Pick two adjacent letters
\(u_0,u_1\) in \(D_1\), and one letter in each of
\[
  D_2,\ldots,D_{c_1-1}.
\]
Let \(\alpha_1\) be the cycle prescribed by
Lemma~\ref{lem:local-rules}(b).  In \(\beta_1\), the cycles
\[
  D_1,\ldots,D_{c_1-1}
\]
have become the fixed point \(u_0\) and one further cycle \(R_1\) on all
their other letters. If \(s>1\), the
cycle \(D_{r_1}\) has contributed only one letter to \(\alpha_1\), so choose
and mark a second letter of \(D_{r_1}\) as the spare letter on \(R_1\).

Suppose \(\alpha_1,\ldots,\alpha_{u-1}\) have been chosen for some
\(2\le u\le s\), such that \(\beta_{u-1}\) has the fixed point created in the
first step and a cycle \(R_{u-1}\) on all letters of
\(D_1,\ldots,D_{r_{u-1}}\) other than \(u_0\). Suppose a marked spare
letter of \(D_{r_{u-1}}\) lies on \(R_{u-1}\) and has not been used by
\(\alpha_1,\ldots,\alpha_{u-1}\). For \(\alpha_u\), use that spare letter and one letter
from each of
\[
  D_{r_{u-1}+1},\ldots,D_{r_u}.
\]
This is exactly \(1+(r_u-r_{u-1})=c_u\) letters.  Let \(\alpha_u\) be the join
cycle from Lemma~\ref{lem:local-rules}(a). It joins the new cycles of
\(\beta\) to \(R_{u-1}\), leaves \(u_0\) fixed, and produces a cycle
\(R_u\) containing all other letters of \(D_1,\ldots,D_{r_u}\).
If \(u<s\), the
cycle \(D_{r_u}\) has just contributed its first letter, so choose and mark a
second, unused letter of \(D_{r_u}\) on \(R_u\) as the spare letter for
the next step.

The construction touches exactly
\[
  (c_1-1)+\sum_{u=2}^s(c_u-1)=a
\]
cycles of \(\beta\), namely \(D_1,\ldots,D_a\).  It uses two
letters from \(D_1\)
and from each connector \(D_{r_u}\) with \(1\le u<s\), and one letter from
every other touched cycle.  Since \(B\) is fixed-point-free, all these choices
can be made with distinct letters.  Thus \(\alpha\) has cycle type \(F\).

\medskip
\noindent\textbf{Case 2: all \(c_u=2\).}
Here \(a=s\).  If \(a=1\), choose any cycle
\(D_1\) of \(\beta\), choose adjacent letters
\(u_0,u_1\in D_1\), and set \(\alpha_1=(u_0\,u_1)\).
Lemma~\ref{lem:local-rules}(b), with \(r=0\), gives two product cycles on
\(D_1\), one of which is the fixed point \(u_0\).  There is no further cycle
\(\alpha_u\) to choose, so the construction is complete.

Now suppose \(a>1\).  Then \(B\) has a part at least \(3\), since otherwise
\(B=[2^n]\) would have gcd \(2\).
Choose \(D_1\) to be a cycle of \(\beta\) of length at least
\(3\), and choose any further \(a-1\) cycles
\(D_2,\ldots,D_a\) of \(\beta\).  Choose adjacent letters
\(u_0,u_1\in D_1\), and set
\(\alpha_1=(u_0\,u_1)\).  The permutation \(\beta_1\) has a fixed point and
one further cycle \(R_1\) on the other letters of \(D_1\). The first transposition uses
two letters of \(D_1\), and \(|D_1|\ge3\) lets us mark a third letter of
\(D_1\) as the spare letter on \(R_1\).

For \(u=2,\ldots,s\), let \(\alpha_u\) use the marked, unused spare letter of
\(D_{u-1}\) and one new letter of \(D_u\).  By
Lemma~\ref{lem:local-rules}(a), this joins \(D_u\) to the cycle
\(R_{u-1}\), leaves the fixed point untouched, and produces a cycle \(R_u\)
on all letters of \(D_1,\ldots,D_u\) other than \(u_0\). If \(u<s\), mark a
second, unused letter of \(D_u\) as the spare letter for the next step.

In both subcases the construction touches exactly \(D_1,\ldots,D_a\).  When
\(a>1\), it uses three letters from \(D_1\), two from each of
\(D_2,\ldots,D_{a-1}\), and one from \(D_a\); when \(a=1\), it uses two
letters from \(D_1\).  The only cycle contributing three letters was chosen
with \(|D_1|\ge3\), and every cycle contributing two letters has length at
least \(2\).  Hence all chosen letters are distinct, and \(\alpha\) has cycle
type \(F\).

The product \(\alpha\beta\) has replaced the \(a\) touched
cycles of \(\beta\) by exactly two cycles: the fixed point created in the
first step and one further cycle on all remaining letters of the touched
cycles of \(\beta\), whose length may be \(1\) only in the case \(a=1\) and
\(|D_1|=2\).  Untouched cycles of \(\beta\) remain unchanged.  Therefore the
total
number of cycles satisfies
\[
  \ell(M)=\ell(B)-a+2.
\]
Consequently,
\[
  b(M)=b(B)+a-2.
\]
The product contains a fixed point, hence \(g(M)=1\).

Since we have constructed one product of type \(M\), conjugating the factors
shows \(\Cl_M\subset \Cl_F\Cl_B\).
\end{proof}

\subsection{Preserving the lcm condition}

\begin{lemma}[Fixed-point defect bound]\label{lem:fixed-lcm-defect}
Let \(G\vdash d\) have at least one fixed point and lcm \(L=m(G)\).  Then
\[
  b(G)\le (d-1)\left(1-\frac1L\right).
\]
In particular, if \(m(G)=2\), then \(b(G)\le (d-1)/2\).
\end{lemma}

\begin{proof}
All non-fixed parts of \(G\) divide \(L\), hence are at most \(L\).  Since
\(G\) has a fixed point,
\[
  \ell(G)\ge 1+\left\lceil\frac{d-1}{L}\right\rceil.
\]
Thus
\[
  b(G)=d-\ell(G)
  \le d-1-\left\lceil\frac{d-1}{L}\right\rceil
  \le (d-1)\left(1-\frac1L\right).
\]
\end{proof}

\begin{lemma}[Excluding an additional lcm-\(2\) branch]\label{lem:no-leftover-lcm2}
Let \(\Lambda\) be a mixed residual datum, and let \((F,B)\) be a mixed
A1-small pair in \(\Lambda\), where \(F\) has fixed points, \(a=b(F)\), and
\(B\) is fixed-point-free gcd-one.  Then no other fixed-point branch has lcm
\(2\).
\end{lemma}

\begin{proof}
Proposition~\ref{prop:standard-residual} gives
\[
  a+b(B)\le d-3
\]
and shows that no pair of fixed-point branches is A1-small.  Together with
\(b(B)=d-\ell(B)\), the displayed inequality gives
\(\ell(B)\ge a+3\).  Since \(B\) is
fixed-point-free, \(\ell(B)\le d/2\), so
\[
  a\le d/2-3.
\]
If another fixed-point branch \(G\) had \(m(G)=2\), then
Lemma~\ref{lem:fixed-lcm-defect} would give \(b(G)\le(d-1)/2\).  Since no
pair of fixed-point branches is A1-small, \(a+b(G)>d-1\), hence \(a>(d-1)/2\), contradicting
\(a\le d/2-3\).
\end{proof}

\begin{lemma}[Separate avoidance of lcm \(2\) and \(3\)]\label{lem:avoid-two-three}
Under the hypotheses of Lemma~\ref{lem:genus-lowering-product}, suppose in
addition that \(n=\ell(B)\ge a+2\).  There is a choice of the
genus-lowering product for which \(m(M)\ne2\), and there is a possibly different
choice for which \(m(M)\ne3\).  These assertions do not require a single choice
that avoids both values.
\end{lemma}

\begin{proof}
In the genus-lowering construction, the \(a\) selected parts of \(B\) are
replaced by a fixed point and a part of size \(S-1\), where \(S\) is their
sum; the unselected parts of \(B\) remain as parts of \(M\).

First avoid \(m(M)=2\).  Since \(B\) is fixed-point-free and gcd-one, it has a
part \(x>2\).  The construction selects exactly \(a\) parts of \(B\), and
\(n\ge a+2\).  Thus \(x\) can be left unselected unless
\(c_1=\cdots=c_s=2\), \(a>1\), and \(x\) is the unique part
of \(B\) greater than \(2\); in that exceptional case, the initial step
requires the cycle corresponding to \(x\).  If \(x\) is left unselected, then
it remains a part of \(M\), so \(m(M)>2\).  In the exceptional case, select
\(x\) and \(a-1\) parts equal to \(2\).  Then
\[
  S-1=x+2(a-1)-1\ge4,
\]
so again \(m(M)>2\).  When \(a=1\), one may select a part different from \(x\)
and leave \(x\) unselected, since \(n\ge3\).

Next avoid \(m(M)=3\).  Since \(g(B)=1\), some part \(y\) of \(B\) is not
divisible by \(3\).  If \(y\) can be left unselected, then it remains a part of
\(M\), and hence \(m(M)\ne3\). The only possible obstruction is that
\(c_1=\cdots=c_s=2\), \(a>1\), and \(y\) is the unique part of
\(B\) greater than \(2\). The initial step must then use its cycle,
but an unselected part equal to \(2\) remains because \(n\ge a+2\).
Thus \(2\) is a part of \(M\), and again \(m(M)\ne3\).
\end{proof}

\begin{lemma}[Preservation of the lcm condition]\label{lem:mixed-lcm-condition}
Let \(\Lambda\) be a positive-genus mixed residual datum in
case~\textnormal{(iii)} of Proposition~\ref{prop:standard-residual}, and let
\((F,B)\) be a mixed A1-small pair in \(\Lambda\), where \(F\) has fixed
points and \(B\) is fixed-point-free.  Put \(a=b(F)\).  Then \((F,B)\)
satisfies the hypotheses of Lemma~\ref{lem:genus-lowering-product}, with
\[
  \ell(B)\ge a+3.
\]
Moreover, the genus-lowering reduction of \((F,B)\) to an output \(M\) can be
chosen so that the resulting datum satisfies the lcm condition.
\end{lemma}

\begin{proof}
Proposition~\ref{prop:standard-residual} gives
\[
  a+b(B)\le d-3.
\]
Since \(b(B)=d-\ell(B)\), it follows that
\[
  \ell(B)\ge a+3.
\]
Thus Lemma~\ref{lem:genus-lowering-product} applies, and the stronger bound
also allows us to use Lemma~\ref{lem:avoid-two-three}.

Every output \(M\) of the genus-lowering construction is
nontrivial.  Indeed, a fixed-point-free gcd-one partition \(B\) has at least
two parts, so \(b(B)\ge2\); since \(a=b(F)\ge1\),
Lemma~\ref{lem:genus-lowering-product} gives
\(b(M)=b(B)+a-2\ge1\).  In particular, \(m(M)\ge2\).

By Lemma~\ref{lem:lcm-obstructions}, the lcm condition is automatic
if at least five branches remain. If four remain, choose an output with
\(m(M)\ne2\) using Lemma~\ref{lem:avoid-two-three}; the only possible
failure is then excluded.

For a reduced triple, apply the same classification.
By Lemma~\ref{lem:no-leftover-lcm2}, an inherited fixed-point branch cannot
have lcm \(2\), while an inherited fixed-point-free branch has lcm at least
\(6\) by Lemma~\ref{lem:fpf-lcm}.  Thus
neither inherited lcm is \(2\).
If the inherited lcms are \(3,3\), the lcm condition fails exactly when
\(m(M)=3\). If the inherited lcms are \(4,4\), the lcm condition fails
exactly when \(m(M)=2\). If the inherited lcms are \(3,6\), the lcm
condition fails exactly when \(m(M)=2\). In all other cases, it holds
for every \(m(M)\ge2\).
The two inherited branches are the original branches outside the reduction
pair \((F,B)\).  Their partitions, and therefore their lcms, are fixed before
the parts of \(B\) used in the construction of
Lemma~\ref{lem:genus-lowering-product} are selected; only the new branch \(M\)
depends on that choice.

If a value of \(m(M)\) causing failure exists, apply the corresponding
avoidance choice
from Lemma~\ref{lem:avoid-two-three}; otherwise any output of
Lemma~\ref{lem:genus-lowering-product} suffices. Thus the lcm condition
is preserved. Only one value needs to be avoided for the two prescribed
inherited branches, so the separate choices in
Lemma~\ref{lem:avoid-two-three} suffice.
\end{proof}

\begin{corollary}[Positive-genus mixed residual reduction]\label{cor:positive}
Let \(\Lambda\) be a mixed residual datum as in
Proposition~\ref{prop:standard-residual}, of source genus \(h\ge1\), and let
\((F,B)\) be a mixed A1-small pair.  Then \(\Lambda\)
reduces to a Zieve-admissible datum with one fewer branch and source genus
\(h-1\).
\end{corollary}

\begin{proof}
Put \(a=b(F)\).
By Lemma~\ref{lem:mixed-lcm-condition}, the genus-lowering output \(M\) may be
chosen so that replacing \(F,B\) by \(M\) preserves the lcm condition.  This
replacement reduces the number of branches by one.  Its change in total
defect is
\[
  a+b(B)-b(M)=a+b(B)-(b(B)+a-2)=2.
\]
The partition \(B\) is fixed-point-free and gcd-one, so it has
at least two parts and \(b(B)\ge2\).  Since \(a\ge1\), the equality
\(b(M)=b(B)+a-2\) gives \(b(M)\ge1\); hence \(M\) is nontrivial.
Lemma~\ref{lem:genus-lowering-product} gives \(g(M)=1\), while all inherited
branches retain the nontriviality and gcd-one properties of the original
datum.  The defect calculation therefore shows that the resulting datum is a
candidate datum of source genus \(h-1\), which is nonnegative because
\(h\ge1\).  Together with the lcm condition, these facts show that the
resulting datum is Zieve-admissible.
\end{proof}

\section{Genus-zero mixed residuals}\label{sec:residual-zero}

This section completes the reduction for mixed residual data of source genus
zero, where the genus-lowering reduction of
Section~\ref{sec:residual-positive} cannot be used.  We first show that such a
datum has exactly one fixed-point branch and exactly three fixed-point-free
branches. We then realize Move A1 by a forest construction and use
a gcd argument to select an output for which the reduced triple is
Zieve-admissible.  The assumed three-branch-point case then
yields realizability of the original datum.

\begin{lemma}[Uniqueness of the fixed-point branch]\label{lem:genus-zero-unique-fixed}
In a genus-zero mixed residual datum of Proposition~\ref{prop:standard-residual},
there is exactly one fixed-point branch.
\end{lemma}

\begin{proof}
Proposition~\ref{prop:standard-residual} gives an A1-small pair \((F,B)\), where
\(F\) has fixed points and \(B\) is fixed-point-free.  Put
\[
  a=b(F),\qquad b=b(B).
\]
Then
\[
  a+b\le d-1.
\]
By Lemma~\ref{lem:fpf-defect}, \(b\ge\lfloor d/2\rfloor+1\).  Hence
\begin{equation}\label{eq:b-a-positive}
  b-a\ge 2b-d+1>0.
\end{equation}

Suppose, for contradiction, that there is another fixed-point branch.  Since no
pair of fixed-point branches is A1-small in the mixed residual
case, every other
fixed-point branch \(G\) satisfies
\begin{equation}\label{eq:no-fp-fp-small}
  a+b(G)\ge d.
\end{equation}

There cannot be another fixed-point-free branch \(B'\): if such a branch
existed, then using \eqref{eq:no-fp-fp-small} for one other fixed branch \(G\),
\[
  a+b+b(G)+b(B')
  \ge d+b+b(B')
  > 2d-2,
\]
again by Lemma~\ref{lem:fpf-defect}.  This already exceeds the genus-zero
total defect \(2d-2\).

Thus \(B\) is the only fixed-point-free branch.  Since the
datum has \(k\ge4\) branches, it then has at least three fixed-point branches.
Choose two of them other than \(F\), and denote them by \(G,H\).  Applying
\eqref{eq:no-fp-fp-small} to both gives
\[
  a+b+b(G)+b(H)
  \ge a+b+2(d-a)
  =2d+(b-a)
  >2d-2
\]
by \eqref{eq:b-a-positive}, again contradicting genus zero.  Therefore no other
fixed-point branch exists.
\end{proof}

\begin{lemma}[Number of branches in genus zero]\label{lem:genus-zero-k4}
A genus-zero mixed residual datum
\[
  (F,B_1,\ldots,B_r)
\]
with one fixed-point branch \(F\) and fixed-point-free gcd-one branches
\(B_i\) must have \(r=3\).  Equivalently, the total number of branches is
\(k=4\).
\end{lemma}

\begin{proof}
Put \(a=b(F)\ge1\).  Genus zero gives
\[
  a+\sum_{i=1}^r b(B_i)=2d-2.
\]
By Lemma~\ref{lem:fpf-defect},
\[
  b(B_i)\ge \left\lfloor\frac d2\right\rfloor+1.
\]
If \(r\ge4\), then the left-hand side is at least
\[
  a+4\left(\left\lfloor\frac d2\right\rfloor+1\right)>2d-2,
\]
for all \(d\ge2\), a contradiction.  A mixed residual datum
has at least four branches, so \(r\ge3\).  Hence \(r=3\).
\end{proof}

By Lemmas~\ref{lem:genus-zero-unique-fixed} and
\ref{lem:genus-zero-k4}, every genus-zero mixed residual datum has the form
\[
  (F,B_1,B_2,B_3).
\]
Writing \(a=b(F)\) and \(n_i=\ell(B_i)\), the Riemann--Hurwitz equation becomes
\[
  a+\sum_{i=1}^3(d-n_i)=2d-2,
\]
or equivalently
\begin{equation}\label{eq:length-sum}
  n_1+n_2+n_3=d+a+2.
\end{equation}

\subsection{A forest construction for Move A1}

\begin{lemma}[Forest construction for Move A1]\label{lem:single-component}
Let
\[
  F=[c_1,\ldots,c_s,1^f]\vdash d,
  \qquad s\ge1,\quad f\ge1,\quad
  c_u\ge2\quad(1\le u\le s),
\]
and put \(a=b(F)\).  Let \(B=[x_1,\ldots,x_n]\vdash d\) be
fixed-point-free.  Suppose
\(n\ge a+1\).  For any subset
\(S\subset\{1,\ldots,n\}\) with \(|S|=a+1\), there is a partition \(M\)
satisfying
\[
  \Cl_M\subset \Cl_F\Cl_B,\qquad b(M)=a+b(B),
\]
with
\[
  M=\left[\sum_{j\in S}x_j,\; x_t\ (t\notin S)\right].
\]
\end{lemma}

\begin{proof}
Put \(C=\sum_{u=1}^s c_u=a+s\) and \(X=\sum_{j\in S}x_j\).
Since \(s\le a\), \(|S|=a+1\), and \(x_j\ge2\),
\[
  X\ge2(a+1)>C,\qquad 0\le X-C\le d-C=f.
\]
Take black vertices with degrees \(c_1,\ldots,c_s\) and \(X-C\)
black degree-one vertices, and take the white vertices indexed by \(S\).
Both degree sums equal \(X\), and the total number of vertices is
\[
  s+(X-C)+(a+1)=X+1.
\]
Lemma~\ref{lem:balanced-tree-component} therefore supplies a tree
component of weight \(X\). For every \(t\notin S\), use a disjoint
star with white centre of degree \(x_t\) and \(x_t\) black degree-one vertices. These stars use exactly the remaining
\[
  \sum_{t\notin S}x_t=d-X=f-(X-C)
\]
black degree-one vertices. By Lemma~\ref{lem:balanced-tree-component} and
Proposition~\ref{prop:incidence-forest}, this forest is realized by
permutations with the prescribed input types and product type
\(M=[X,x_t\ (t\notin S)]\). Since \(\ell(M)=n-a\), its defect is
\(b(B)+a\), as required.
\end{proof}

For the complement \(T=\{1,\ldots,n\}\setminus S\), put
\[
  q=|T|=n-a-1.
\]
The product type in Lemma~\ref{lem:single-component} may be written as
\[
  M_T=
  \left[d-\sum_{t\in T}x_t,\; x_t\ (t\in T)\right].
\]
Therefore
\begin{equation}\label{eq:gcd-complement}
  g(M_T)=\gcd\bigl(d,\;x_t\ (t\in T)\bigr).
\end{equation}

\begin{proposition}[Genus-zero mixed residual reduction]\label{prop:genus-zero}
Every genus-zero mixed residual datum of the form
\[
  (F,B_1,B_2,B_3)
\]
is realizable, assuming the three-branch-point hypothesis of
Theorem~\ref{thm:main}.
\end{proposition}

\begin{proof}
Put \(a=b(F)\) and \(n_i=\ell(B_i)\).  We first prove that every pair
\((F,B_i)\) is A1-small.  By Lemma~\ref{lem:fpf-defect},
\[
  n_i=d-b(B_i)\le
  \begin{cases}
    d/2-1, & d\text{ even},\\
    (d-1)/2, & d\text{ odd}.
  \end{cases}
\]
Hence \(n_j+n_k\le d-1\) for distinct \(j,k\). For
\(\{i,j,k\}=\{1,2,3\}\), the genus-zero length equation
\eqref{eq:length-sum} therefore gives
\[
  n_i=d+a+2-(n_j+n_k)\ge a+3.
\]
In particular, every pair \((F,B_i)\) is A1-small. Define
\[
  q_i=n_i-a-1\ge2\qquad(i=1,2,3).
\]

By \eqref{eq:length-sum}, at least one \(B_i\) has a part \(2\):
otherwise all parts of all three branches would be at least \(3\),
giving \(n_1+n_2+n_3\le d\), contrary to
\(n_1+n_2+n_3=d+a+2\). Fix such an \(i\), and write
\(B_i=[x_1,\ldots,x_{n_i}]\) with \(x_1=2\).

If \(d\) is odd, the singleton \(\{1\}\) satisfies
\(\gcd(d,x_1)=1\). If \(d\) is even, \(g(B_i)=1\) guarantees an
odd part \(x_v\), so
\[
  \gcd(d,x_1,x_v)=\gcd(d,2,x_v)=1.
\]
In either case, enlarge these one or two indices to a subset
\(T\subseteq\{1,\ldots,n_i\}\) of size \(q_i\); this is possible
because \(2\le q_i\le n_i\). Adding indices preserves gcd one.
Lemma~\ref{lem:single-component} and \eqref{eq:gcd-complement} now give
an output \(M=M_T\) satisfying
\[
  \Cl_M\subset\Cl_F\Cl_{B_i},\qquad g(M)=1,
  \qquad b(M)=a+b(B_i).
\]
Let \(j,k\) be the other two indices.
Replacing \(F,B_i\) by \(M\) gives the triple
\((M,B_j,B_k)\).  Its total defect is
\[
  b(M)+b(B_j)+b(B_k)
  =a+\sum_{r=1}^3 b(B_r)
  =2d-2.
\]
All three branches are gcd-one.  They are also nontrivial:
the branches \(B_j,B_k\) are fixed-point-free, while
\(b(M)=a+b(B_i)>0\).  Hence \(m(M)\ge2\), and
Lemma~\ref{lem:fpf-lcm} gives \(m(B_j),m(B_k)\ge6\).  Consequently,
\[
  \frac1{m(M)}+\frac1{m(B_j)}+\frac1{m(B_k)}
  \le \frac12+\frac16+\frac16
  =\frac56<1.
\]
For a triple, the lcm condition is equivalent to requiring
that this reciprocal sum differ from \(1\), so the reduced triple satisfies
that condition.  It is therefore Zieve-admissible.  The three-branch-point hypothesis
of Theorem~\ref{thm:main} realizes the reduced triple, and
Proposition~\ref{prop:expansion} then realizes the original datum.
\end{proof}

\section{Proof of the main theorem}\label{sec:main-proof}

\begin{proof}[Proof of Theorem~\ref{thm:main}]
Fix \(d\ge3\), the case \(d=2\) being vacuous. We use strong induction on
the number \(k\ge3\) of branches, uniformly over all nonnegative source
genera. The case \(k=3\) is precisely the hypothesis of
Theorem~\ref{thm:main}. Let \(k\ge4\), and assume that every degree-\(d\)
Zieve-admissible datum with \(j\) branches is realizable for all
\(3\le j<k\), regardless of its source genus.

Let \(\Lambda\) be an arbitrary degree-\(d\) Zieve-admissible datum
with \(k\) branches. Apply Proposition~\ref{prop:standard-residual},
whose alternatives are labeled
\textnormal{(i)}--\textnormal{(iii)}.  In
case~\textnormal{(i)}, let \(D\) be a reduced candidate datum supplied by that
alternative.  The datum \(D\) satisfies the hypotheses of
Corollary~\ref{cor:special-reduced-data} and is therefore realizable.
Proposition~\ref{prop:expansion} then realizes the original datum
\(\Lambda\).  In case~\textnormal{(ii)}, Move A1 produces a Zieve-admissible
datum with \(k-1\) branches.  The strong induction hypothesis
realizes this reduced datum,
and Proposition~\ref{prop:expansion} again realizes \(\Lambda\).

It remains to consider case~\textnormal{(iii)}, the mixed
residual case.  Every mixed A1-small pair \((F,B)\) then satisfies
\[
  b(F)+b(B)\le d-3.
\]
If the source genus is positive, Corollary~\ref{cor:positive} reduces
\(\Lambda\) to a Zieve-admissible datum with one fewer branch.  The strong
induction hypothesis realizes the reduced datum, and
Proposition~\ref{prop:expansion} realizes \(\Lambda\).  If the source genus is
zero, Proposition~\ref{prop:genus-zero} realizes \(\Lambda\) directly under the
three-branch-point hypothesis.  Thus every alternative in
Proposition~\ref{prop:standard-residual} leads to realizability of \(\Lambda\),
which completes the induction.

\end{proof}

\subsection*{Acknowledgements} 

B.X. is supported in part by the Project of Stable Support for Youth Team in Basic Research Field, CAS (Grant No. YSBR-001) and the National Natural Science Foundation of China (Grant No. 12271495). 
Y.Y. is partially supported by the National Key R\&D Program of China (No. 2024YFA1013802), the National Natural Science Foundation of China (Nos. 12131015 and 12371042), and the Quantum Science and Technology-National Science and Technology Major Project (No. 2021ZD0302902).
Y.Z. is supported in part by the Anhui Province Science and Technology Tackle Plan Project (No. 202423k09020008).

\subsection*{Declaration of generative AI and AI-assisted technologies
in the manuscript preparation process}

OpenAI's GPT-5.6 Sol and GPT-6 Astra assisted throughout the research and
preparation of this article, including the exploration of reduction strategies,
proof development and checking, consistency checks on notation and terminology,
and manuscript drafting and revision.  The authors reviewed all AI-assisted
material with careful scrutiny and take full responsibility for both all mathematical claims and the
final text.

\bibliographystyle{amsplain}
\bibliography{refs}

\end{document}